\documentclass[11pt]{amsart}
\usepackage{amsmath,amsthm,amssymb,amsfonts,amscd}
\usepackage{amsxtra}
\usepackage{eucal}
\usepackage{mathrsfs}
\usepackage[pdftex]{color, graphicx}
\usepackage[all]{xy}
\usepackage{hyperref}
\usepackage{cases}
\usepackage{pgf,tikz}
\usetikzlibrary{shapes,arrows}
\hypersetup{
  colorlinks,
  citecolor=blue,
  linkcolor=brown,
  urlcolor=magenta}
  
\headsep=1cm \numberwithin{equation}{section}
\SelectTips{cm}{11}

\newcommand \Proj {\ensuremath{\mathrm{Proj}}}
\newcommand \im   {\ensuremath{\mathrm{im}}}

\newcommand \ini {\ensuremath{\mathrm{in}}}
\newcommand \Tor {\ensuremath{\mathrm{Tor}}}

\newcommand \reg {\mathrm {reg}}

\newcommand \codim {\ensuremath{\mathrm{codim}}}

\newcommand \depth {\ensuremath{\mathrm{depth}}}

\def\P{{\mathbb P}}
\def\Z{{\mathbb Z}}
\def\N{{\mathbb N}}
\def\T{{\mathit T}}

\newcommand \st[1] {\stackrel{#1}{\longrightarrow}}
\newcommand \sts[1] {\stackrel{#1}{\rightarrow}}

\newcommand \lra {\rightarrow}

\newcommand\keyword[1]{\par{\bf {\it Keywords:~}}{#1}}

\newcommand \rrlap[1]{\hbox to 0pt{#1}}

\theoremstyle{plain} 
\newtheorem{Thm}{Theorem}[section]
\newtheorem{Prop}[Thm]{Proposition}

\newtheorem{Lem}[Thm]{Lemma}

\theoremstyle{definition}
\newtheorem{Def}[Thm]{Definition}
\newtheorem{Ex}[Thm]{Example}
\newtheorem{Qu}[Thm]{Question}

\newtheorem{Remk}[Thm]{Remark}

\newtheorem{Obs}[Thm]{Observation}
\newtheorem{Fact}[Thm]{Fact}

\begin{document}

\title[Sharp bounds for higher linear syzygies]{Sharp bounds for higher linear syzygies and classifications of projective varieties}
\author[K.\ Han and S.\ Kwak]{Kangjin Han and Sijong Kwak}
\address{School of Mathematics, Korea Institute for Advanced Study (KIAS),
85 Hoegiro, Dongdaemun-gu, Seoul 130--722, Korea}
\email{kangjin.han@kias.re.kr}
\address{Department of Mathematics, Korea Advanced Institute of Science and Technology (KAIST),
373-1 Guseong-dong, Yuseong-gu, Daejeon, Korea}
\email{skwak@kaist.ac.kr}
\thanks{
The first author was supported by Basic Science Research Program through the National Research Foundation
of Korea (NRF) funded by the Ministry of Education, Science and Technology (grant no. 2012R1A1A2038506). The second author was supported by Basic Science Research Program through the National Research Foundation
of Korea (NRF) funded by the Ministry of Education, Science and Technology (grant no. 2009-0063180)}
\date{\today}


\begin{abstract}
In the present paper, we consider upper bounds of higher linear syzygies i.e. graded Betti numbers in the first linear strand of the minimal free resolutions of projective varieties in arbitrary characteristic. For this purpose, we first remind `Partial Elimination Ideals (PEIs)' theory and introduce a new framework in which one can study the syzygies of embedded projective schemes well using PEIs theory and the reduction method via inner projections.

Next we establish fundamental inequalities which govern the relations between the graded Betti numbers in the first linear strand of an algebraic set $X$ and those of its inner projection $X_q$. Using these results, we obtain some natural sharp upper bounds for higher linear syzygies of any nondegenerate projective variety in terms of the codimension with respect to its own embedding and classify what the extremal case and next-to-extremal case are. This is a generalization of Castelnuovo and Fano's results on the number of quadrics containing a given variety and another characterization of varieties of minimal degree and del Pezzo varieties from the viewpoint of `syzygies'. Note that our method could be also applied to get similar results for more general categories (e.g. connected in codimension one algebraic sets).
 
\bigskip
\noindent\keyword{linear syzygies, graded Betti numbers, property $\textbf{N}_{d,p}$, partial elimination ideals, inner projection, varieties of small degree.}
\end{abstract}

\maketitle
\tableofcontents \setcounter{page}{1} 
\section{Introduction}\label{intro}

Let $X\subset\P^{N}$ be any nondegenerate variety (i.e. irreducible and reduced closed subscheme) over any field $k$ of dimension $n$ and of $\codim(X,\P^N)=e$. Let $R:=k[x_0,\cdots,x_N]$ be the coordinate ring of $\P^N$ and $R_X:=R/I_X$ be also the one of $X$. The \textit{graded Betti numbers} of $X$ is defined by
\begin{equation}
\beta_{p,q}(X):=\dim_k\Tor^R_p(R_X,k)_{p+q}
\end{equation}

and the \textit{Betti table} of $X$, $\mathbb{B}(X)$ consists of these graded Betti numbers of $X$. This table is usually considered to represent the type of the minimal free resolution of $R_X$. For instance, $\beta_{1,1}$ corresponds to the number of (independent) \textit{quadrics} containing $X$ and so does $\beta_{2,1}$ to the \textit{linear} syzygies on them. We may present $\mathbb{B}(X)$ typically as follows:
{\setlength\abovecaptionskip{.1ex plus .125ex minus .125ex}
\begin{figure}[!hbp]
\begin{align*}\small
\mathbb{B}(X)&\quad
\begin{array}{c|c|c|c|c|c|c|c|c|c|c|c|c|c|}
  &0 &1 &\cdots &a(X) &a+1 &\cdots & b-1 & b(X) &\cdots  & p & \cdots \\ \hline
0 &1 & - &\cdots &- &- &\cdots & - & - &\cdots & - & \cdots \\ \hline
1 &- &\beta_{1,1} &\cdots & \beta_{a,1} & \beta_{a+1,1} &\cdots & \beta_{b-1,1} & - &\cdots &- & \cdots \\ \hline
2 &- &- &\cdots & - & \beta_{a+1,2} &\cdots & \beta_{b-1,2} & \beta_{b,2} &\cdots & \beta_{p,2} & \cdots \\ \hline
\vdots &- &- &\cdots & - & \ddots &\cdots & \vdots & \vdots  &\cdots & \ddots  & \cdots \\ \hline
q &- &- &\cdots & - & \beta_{a+1,q} &\cdots & \beta_{b-1,q} & \beta_{b,q}  &\cdots & \beta_{p,q}  & \cdots \\ \hline
\vdots &- &- &\cdots & - & \ddots &\cdots & \vdots & \vdots  &\cdots & \ddots  & \cdots \\ \hline
\end{array}
\end{align*}
\caption{\textsf{Betti table} of a nondegenerate variety $X$ in $\P^N$. We denote \textit{zero} by $-$. By two pivotal places, determined by $a=a(X), b=b(X)\ge 0$, we could characterize the \textit{first} linear strand of this resolution.}
\label{fig_betti}
\end{figure}}

Since M. Green showed through his foundational paper \cite{G2} several results which imply some of strong connections between geometry of projective varieties and their syzygies, there have been many problems and conjectures concerning shapes of $\mathbb{B}(X)$ and structures on some or all of $\{ \beta_{p,q}\textrm{'s}\}$. In this paper we will consider some interesting problems based on the first linear strand of Betti tables of projective varieties (or schemes) particularly.

By convention, we call the subcomplex (or the corresponding part of table) represented by Betti numbers $\beta_{1,1},\cdots,\beta_{b-1,1}$ in the second row of $\mathbb{B}(X)$ \textit{the (first) linear strand of $\mathbb{B}(X)$}. Following the notations in \cite{Eis}, we also denote the (homological) index to which the resolution admits \textit{only} linear syzygies by $a(X)$ and the first index from which there exists \textit{no more} linear syzygy by $b(X)$. Then, the linear strand of the minimal free resolution of $R_X$ can be characterized by these invariants $a(X)$ and $b(X)$.

Classically, there have been well-known results on the number of quadratic equations containing $X$, i.e. $\beta_{1,1}(X)$ (see \cite{L'vo,Z,HK1} for modern references). Before stating them, let us make our terminology clear. Say $d=\deg(X)$, \textit{degree} of $X$. One can say that $X$ is a \textit{variety of minimal degree} (abbr. VMD) if $d=e+1$. Here we call $X$ \textit{of next-to-minimal degree} when $d=e+2$. Furthermore, throughout this paper, we call $X$ a \textit{del Pezzo} variety if $X$ is \textit{arithmetically Cohen-Macaulay} (abbr. ACM) and of next-to-minimal degree. Then, the theorems say\\

\begin{itemize}
\item[(a)] {\it(Castelnuovo, 1889) Let $X$ be as above, \[\beta_{1,1}(X)\le {e+1 \choose 2}\] and the $``="$ holds if and only if $X$ is a variety of minimal degree.}
\item[(b)] {\it(Fano, 1894) Unless $X$ is a variety of minimal degree,\[\beta_{1,1}(X)\le {e+1 \choose 2}-1\] and the $``="$ holds if and only if $X$ is a del Pezzo variety.}
\end{itemize}

But when we move on higher $p$'s, it is not so feasible to handle \textit{higher linear syzygies (i.e. $\beta_{p,1}(X)$'s)} directly as to manipulate them in case of $p$ being very low (e.g. considering generators, their relations, and so on). In this paper we introduce a useful way to treat higher linear syzygies in a quite effective manner, that is

\begin{center}
 \textit{Projecting higher linear syzygies of $X$ to those of its projected image $X_{\tt{q}}$.}
\end{center} 
 
Especially, we will focus on \textit{inner} projection process (i.e. a projection taking its center from inside of $X$) here (see Remark \ref{red_inner} for details). We denote the Zariski closure of the image of $\pi_{\texttt{q}}:X\setminus\{\texttt{q}\}\rightarrow\P^{N-1}$ by $X_{\tt{q}}$. Note that this inner projection process often transplants much of favorable structures on syzygies and Betti table into its projected image, in contrast with outer projection (e.g. see \cite{HK1}).\\

\paragraph*{\textbf{Main results}} Now we present our main results. First, we are giving a very useful inequality through which we can explain the relations between the Betti numbers in the first linear strand of $X$ and $X_{\tt{q}}$ essentially.

\begin{Thm}\label{fund_ineq}
Let $X^n\subset \mathbb P^{n+e}$ be a nondegenerate variety of codimention $e$, $\texttt{q}\in X$ be any closed point of $X$. For any $p\ge 1$, the following holds
\begin{align}\label{ubd_main}
\beta_{p,1}(X)&\le\beta_{p,1}(X_{\texttt{q}})+\beta_{p-1,1}(X_{\texttt{q}})+{e\choose p}~.
\end{align}
\end{Thm}
In fact, the inequality (\ref{ubd_main}) is stated here in a weaker form of its own for simplicity. We will present and prove a more strengthened version of Theorem \ref{fund_ineq} in section \ref{prf_mains} (see Theorem \ref{1st_str_ineq}) for the sake of future use.\\

As a direct consequence of Theorem \ref{fund_ineq}, we can obtain optimal upper bounds on $\beta_{p,1}$ of \textit{every} variety for more higher $p$ in the linear strand.

\begin{Thm}\label{sharp_ubd}
Let $X^n\subset\P^{n+e}$ be any nondegenerate variety of codim $e\ge 1$. Then,
\begin{equation}\label{extr_bd}
\beta_{p,1}(X)\leq p{e+1 \choose p+1}\quad \textrm{for}~all~p\ge 0
\end{equation}
\end{Thm}

Note that $p{e+1 \choose p+1}$ is the $p$-th Betti number of varieties of minimal degree (VMD) of codimension $e$.\\

And we can also add new characterizations to classical ones of VMD's as follows:

\begin{Thm}\label{char_VMD}
Let $X^n\subset \P^{n+e}$ be a nondegenerate variety with $e\ge 1$. Then, the following are all equivalent:
\begin{itemize}
\item[(a)] $X^n$ is a variety of minimal degree in $\P^{n+e}$.
\item[(b)] $\mathcal I_X$ is 2-regular.
\item[(c)] $a(X)\ge e$.
\item[(d)] $h^0(\P^{n+e}, \mathcal I_X(2))$=${e+1 \choose 2}$.
\item[(e)] one of $\beta_{p,1}(X)\textrm{'s}$ achieves the maximum for some $1\le p\le e$.
\item[(f)] all $\beta_{p,1}(X)$ achieve the maxima.
\end{itemize}
\end{Thm}
\paragraph*{\textbf{Organization of the paper}}
For this purpose, we remind Partial Elimination Ideals (PEIs) theory, give account for its relevance to the theory of projections of projective schemes briefly, and introduce a new framework in which one can study syzygies of projective subschemes using PEIs theory and reduction method via inner projections in section \ref{sec_pei}. In section \ref{prf_mains} we give proofs of our main results and add a remark which give some inspiration on how to carry out the computations of Betti numbers using projections. In section \ref{n_extr} we treat \textit{next-to-extremal case} which is a natural generalization of Fano's classical theorem as our previous theorems did for Castelnuovo's. Finally, we give examples and questions to improve our results into more general categories and more refined bounds in section \ref{ex_qu}.\\

\paragraph*{\textbf{Notations and Conventions}}
We are working on the following conventions:
\begin{itemize}
\item (Betti numbers) For any commutative ring $A$ and a graded $A$-module $M$, we also define \textit{graded Betti numbers of $M$}, $\beta^A_{p,q}(M)$ by $\dim_k\Tor^{A}_p (M,k)_{p+q}$. For a polynomial ring $R$ and its homogeneous ideal $I$, we remind an easy fact
$$\Tor^R_p (R/I,k)_{p+q}=\Tor^R_{p-1}(I,k)_{p-1+q+1}~\textrm{for any $p\ge 1,~q\ge 0$}$$
so that $\beta^R_{p,q}(R/I)=\beta^R_{p-1,q+1}(I)$.
We'll write $\beta_{p,q}(M)$ or $\beta_{p,q}$ instead of  $\beta^R_{p,q}(M)$ where it leads no confusion and denote $\beta_{p,q}(R_X)$ simply by $\beta_{p,q}(X)$.
\item (Property $\textbf{N}_{d,p}$) For a homogeneous ideal $I\subset R$, we say that $I$ \textit{satisfies} property $\textbf{N}_{d,p}$ if every $\beta_{i,j}(I)=0$ for any $0\le i<p$ and any $j>d$ (see also \cite{EGHP1,HK1}).
\item (Tor modules) From now on, we often abbreviate $\Tor^A_p(M,k)_{p+q}$ as $\T^A_{p,q}(M)$ for any commutative ring $A$ and a graded $A$-module $M$.
\item (Arithmetic depth) When we refer the \textit{depth of $X$}, denoted by $\depth_R(X)$, we mean the arithmetic depth of $X$, i.e. $\depth_{R}(R_X)$.
\item (Nondegeneracy) Throughout the paper, the \textit{nondegenerate} condition on a scheme $X$ defined by $I$ just means that $I$ has \textit{no linear forms}.
\end{itemize}

\section{Partial elimination ideals (PEIs) and its application}\label{sec_pei}

M. Green introduced partial elimination ideals (PEIs) in his lecture note \cite{G} to study lexicographic generic initial ideals (gins) and subsequent works concerning lex-gins have been done by some authors (see \cite{CS,A08,AKS1}). In this section we will briefly review PEIs theory and try to investigate another application of it. We will also recall some basic facts of the theory which are essential for the remaining part of the paper throughout this section.

\subsection{A brief review of Partial Elimination Ideals}

Let $S=k[x_1,\ldots,x_N]\subset R=k[x_0,x_1\ldots,x_N]$ be two polynomial rings and $I$ be a homogeneous ideal of $R$. For the degree lexicographic order, if $f\in I_m$ has leading term $\ini(f)=x_0^{d_0}\cdots x_N^{d_N}$, we set $d_0(f)=d_0$, the leading power of $x_0$ in $f$. Then we can give the definition of partial elimination ideals of $I$ as follows:

\begin{Def}\label{PEI}
Let $I\subset R$ be a homogeneous ideal and let us define
\[\widetilde{K}_i(I):=\left(\bigoplus_{a=0}^i S\cdot x_0^a\right)\cap I=\bigoplus_{m\ge 0}\big\{f\in I_{m}\mid d_0(f)\leq i\big\}.\] If $f\in \widetilde{K}_i(I)$, we may write
uniquely $f=x_0^i\bar{f}+g$ where $d_0(g)<i$. Now we consider the ideal $K_i(I)$ in $S$ generated by the image of $\widetilde{K}_i(I)$
under the map $f\mapsto \bar{f}$ and we call $K_{i}(I)$ the \textit{$i$-th partial elimination ideal of $I$ with respect to $x_0$}. We define $\widetilde{K}_i(I)$ (so, also $K_{i}(I)$) as zero for any $i<0$ by convention.
\end{Def}

\begin{Obs} We could observe some properties of these
ideals.
\begin{itemize}

\item[(a)] (Finiteness) $\widetilde{K}_i(I)$ is always a finitely generated graded $S$-module (even though $I$ and $R/I$ might not be) and $K_i(I)$ is a homogeneous ideal of $S$.
\item[(b)] $0$--th partial elimination ideal $K_{0}(I)$ of $I$ is equal to
$$
\widetilde{K}_0 (I)=S\cap I=\bigoplus_{m\ge 0}\big\{f \in I_m \mid d_0(f)=0\big\}~
$$
, the \textit{complete} elimination ideal of $I$ with respect to $x_0$.
\item [(c)] (Stabilization) Since $\widetilde{K}_{i}(I)$'s form a natural filtration of $I$ with respect to $x_0$, they induce an ascending chain of $K_i(I)$'s such as:
\[(0)=\widetilde{K}_{-1}(I)\subset \widetilde{K}_0 (I)\subset \widetilde{K}_1(I)\subset \cdots \subset \widetilde{K}_s(I)\subset \widetilde{K}_{s+1}(I)\subset\cdots  \subset R \]
\[(0)={K}_{-1}(I)\subset K_0(I)\subsetneq K_1(I)\subsetneq \cdots \subsetneq K_{s}(I)=K_{s+1}(I)=\cdots\subset
S\]
, where the ascending chain of $K_i(I)$'s is always stabilized in finite steps. Let's define \textit{the stabilization number} $s(I)$, and \textit{the stabilized partial elimination ideal} $K_{\infty}(I)$ as below:
\begin{displaymath}s(I):=\min\{i\in\N|~K_i(I)=K_{i+1}(I)=\cdots\},\quad K_{\infty}(I):=K_s(I)~.
\end{displaymath}
We also define $\widetilde{K}_{\infty}(I):=I$ as $S$-module.
\item[(d)] (Exact sequences) For any $i\in\mathbb{Z}$, there are
two short exact sequences of graded $S$-modules such as

\begin{equation}\label{ses_PEI}
0\rightarrow \frac{\widetilde{K}_{i-1}(I)}{\widetilde{K}_{h}(I)}
\st{incl.}
\frac{\widetilde{K}_{i}(I)}{\widetilde{K}_{h}(I)}\st{f}
K_{i}(I)(-i)\rightarrow 0 
\end{equation}

for every $h\le i-1$ and

\begin{equation}\label{ses_PEI_2}
0\rightarrow \frac{\widetilde{K}_{i-1}(I)}{\widetilde{K}_{i-2}(I)}
(-1)\st{\times x_0}
\frac{\widetilde{K}_{i}(I)}{\widetilde{K}_{i-1}(I)}\st{g}
\frac{K_{i}(I)}{K_{i-1}(I)}(-i)\rightarrow 0~.
\end{equation}

\end{itemize}
\end{Obs}

Using the syzygies of $\widetilde{K}_i (I)$, we can approximate $S$-module syzygy structures of $I$ (more generally, of $I/\widetilde{K}_h (I)$).

\begin{Prop}[Approximation of syzygies]\label{approx}
For given any $p,q\ge 0$ and $h\in\Z$, we have
$$\Tor^S_p(I/\widetilde{K}_h (I),k)_{p+q}\simeq\Tor^S_p (\widetilde{K}_d (I)/\widetilde{K}_h (I),k)_{p+q}$$
for any $d\in\Z$ such that $d\ge q-1$ and $d\ge h$. In particular, if we set $h<0$, then
$$\Tor^S_p(I,k)_{p+q}\simeq\Tor^S_p(\widetilde{K}_d (I),k)_{p+q}$$
holds for any $d\ge q-1$.
\end{Prop}
\begin{proof}
It comes from the definition of $\widetilde{K}_i(I)$ and exact sequence (\ref{ses_PEI}) directly.
\end{proof}
As a consequence, we get a simple, but frequently used lemma.
\begin{Lem}\label{trivial_van}
$\Tor^S_p(I/\widetilde{K}_{h}(I),k)_{p+q}=0$ for every $p\ge 0$ and any $q\le h+1$.
\end{Lem}
\begin{proof}
It is straightforward from Proposition (\ref{approx}).
\end{proof}

\subsection{Applications to projection mappings}

Geometrically, PEIs are closely related to projection mappings of schemes by nature. Consider our scheme $X\subset\P^N=\Proj(R)$ defined by a homogeneous ideal $I\subset R$, take a closed point $\texttt{q}$ of $\P^N$ as centre of our projection. Let $X_{\tt{q}}$ be its image of the projection map $\pi_{\tt{q}}:X\setminus{\{q\}}\rightarrow\P^{N-1}=\Proj(S)$ if $\texttt{q}\notin X$ and be the Zariski closure of the image if $\texttt{q}\in X$.

We define the \textit{partial elimination ideals of $I$ with respect to $\tt{q}$} (denoted by $K_i(\texttt{q},I)$) by the PEI $K_i(I)\textrm{'s}$ of $I$ with respect to $x_0$ assuming $\texttt{q}=(1:0:\cdots:0)$ by a suitable linear change of coordinates. This definition makes sense, because we may define \textit{coordinate-free} version of PEIs with no much difficulty (e.g. \cite{Kur}) and could show that taking these PEIs commutes with coordinate transformations. We often denote $K_i(I)$ and $s(\texttt{q})$ (or just $s$) simply instead of $K_i(\texttt{q},I)$ and $s(\texttt{q},I)$ where no confusion occurs.

Now, let's regard the PEIs of $I$ with respect to $\tt{q}$. First of all, the $0$-th partial elimination ideal $K_{0}(I)=I\cap S$ gives a natural scheme structure on $X_\texttt{q}$ itself. Further, from higher partial elimination ideals we could extract more information on the given projection $\pi_{\tt{q}}$. For \textit{outer} projection case (i.e. $\texttt{q}\notin X$), they turned out to be related \textit{multiple loci} of $\pi_{\texttt{q}}$ (see \cite{G,Hai,CS}). Here, we introduce an extended version including inner projection case also (see \cite{Han}).

\begin{Prop}\label{geo_PEIs} Let $I$ be a homogeneous ideal of $R$ defining $X\subset\P^N$ as a scheme. and let $\mathscr{M}_{i+1}(\pi_{\tt{q}})$ be the multiple loci in $\P^{N-1}$ where each fiber of $\pi_{\tt{q}}$ is a finite scheme of length at least $i+1$. Set-theoretically, we have
\begin{displaymath}
Z(K_{\infty}(I))\cup \mathscr{M}_{i+1}(\pi_{\tt{q}})=Z(K_i(I))~.
\end{displaymath}
\end{Prop}

Thus it is important to see \textit{when the $K_i(I)$'s are stabilized} (i.e. the stabilization number $s(\texttt{q})$) and \textit{what they do look like} (i.e. the stabilized ideal $K_\infty (I)$) for studying of projections. In general, we can give bounds for $s(\texttt{q})$ in terms of degrees of generators and the $K_\infty (I)$ matches an interesting geometric notion in inner projection case as the following proposition says.

\begin{Prop}\label{prop_stab}
Let $X\subset \P^N$ be a projective subscheme with a defining ideal $I$ and $\texttt{q}=(1,0,\ldots,0)\in\P^N$. Suppose that $I$ is generated by homogeneous polynomials of degree at most $d$.
\begin{itemize}
\item[(a)]  Outer case (i.e. $\texttt{q}\notin X$): $s=s(\texttt{q})\le d$ and $K_\infty(I)=(1)$.
\item[(b)] Inner case (i.e. $\texttt{q}\in X$): $s=s(\texttt{q})\le d-1$ and $K_\infty(I)=I_{TC_{\texttt{q}}X}$, where $I_{TC_\texttt{q}}X$ is the ideal of projective tangent cone of $X$ at $\texttt{q}$. In particular, if $\texttt{q}$ is smooth, $K_{d-1}(I)$ consists of linear forms which defines the projective tangent space, $T_\texttt{q} X$.
\end{itemize}
\end{Prop}
\begin{proof} (a) comes from a fact, i.e. there always exists a homogeneous $f\in I$ with its leading term $\ini(f)=x_0^\nu$ and $\nu\le d$. For (b), see proposition 2.5 in \cite{HK1}.\end{proof}

\begin{Remk}[Computations of Betti numbers using PEIs] Using Proposition \ref{prop_stab}, we could compute some pieces of syzygies of an infinitely generated $S$-module $I$ (or more generally, of $I/\widetilde{K}_h (I)$). For any $p\ge0$, $\beta^S_{p,q}(I/\widetilde{K}_h (I))$ is zero for every $q\le h+1$ (Lemma \ref{trivial_van}).  When $q=h+2$, by Proposition \ref{approx} and a short exact sequence (\ref{ses_PEI}) this is equal to 
\begin{equation}
\beta^S_{p,q}(\widetilde{K}_{h+1} (I)/\widetilde{K}_h (I))=\beta^S_{p,q}(K_{h+1}(I)(-h-1))=\beta^S_{p,1}(K_{h+1}(I))
\end{equation}
and it can be computed by the Koszul resolution of the (independent) linear forms in $K_{h+1}(I)$.

In particular, when $I$ is generated in degree $d$, $h=d-2$, and $\texttt{q}\in X=V(I)$ is smooth,  by Proposition \ref{prop_stab},
\begin{equation}
K_{i}(I)=(\ell_1,\ldots,\ell_e)=:I_L\quad\textrm{for every $i\ge d-1$}
\end{equation}
, where $e=N-\dim_k T_{\texttt{q}}X$ and $I_L$ defines the projective tangent space $T_{\texttt{q}}X$. Hence, an infinitely generated $S$-module $I/\widetilde{K}_{d-2} (I)$ has a rather simple minimal free $S$-resolution such as:
\[
0 \to\bigoplus_{q=0}^\infty S(-d-e+1-q)^{b_{e-1}}\to\cdots\to\bigoplus_{q=0}^\infty S(-d-1-q)^{b_1}\to\bigoplus_{q=0}^\infty S(-d-q)^{b_0}\to I/\widetilde{K}_{d-2} (I)
\]
, where $\beta^S_{p,q}(I/\widetilde{K}_{d-2} (I))=b_p=\displaystyle{e\choose p+1}$.
\end{Remk}

\subsection{Syzygies of inner projections}\label{syz_inner}

In this subsection, we explain how we compare the graded Betti numbers of $X$ with those of $X_{\texttt{q}}$ and give some general rules for behaviors of Betti tables under inner projections. First, we recall a mapping cone construction as follows (see e.g. \cite{HK1}):
\begin{Prop}[Elimination mapping cone sequence]\label{MC_seq}
Let $S=k[x_1,\ldots,x_N], R=k[x_0,x_1\ldots,x_N]$ be two polynomial rings. Let $M$ be a graded $R$-module which is not necessarily finitely generated. Then, we have a natural long exact sequence:
\begin{equation*}\small
\begin{array}{cccccccc}
\cdots\Tor_{p}^R(M,k)_{p+q}\lra
\Tor_{p-1}^{S}(M,k)_{p-1+q}
\sts{\bar{\mu}}\Tor_{p-1}^{S}(M,k)_{p-1+q+1}\lra\Tor_{p-1}^R(M,k)_{p-1+q+1}\cdots
\end{array}
\end{equation*}
whose connecting homomorphism $\bar{\mu}$ is induced by the multiplicative map $\times x_0$.
\end{Prop}

Using elimination mapping cone sequence (EMCS) and Betti number calculations of PEIs, we could put Betti numbers of $X$ and those of $X_q$ together in a diagram and relate them each other (see Figure \ref{howto} below).

\begin{figure}[!hbtp]
\xymatrix{
&&\beta_{p,q}^S(X_{\texttt{q}})\ar[r]\ar@{<-->}[dr]&\beta_{p,q}^S(X)\ar[r]\ar[d]^{\textsf{EMCS}}&\quad\beta_{p,q}^S(X)/\beta_{p,q}^S(X_{\texttt{q}}) \qquad\qquad\\
&&&\beta_{p,q}^R(X)\ar@{<-->}[dl]\ar[d]^{\textsf{EMCS}}&\textrm{(from Betti $\#$ of PEIs)}\ar@{~>}[d]\ar@{~>}[u]\\
&&\beta_{p-1,q}^S(X_{\texttt{q}})\ar[r]&\beta_{p-1,q}^S(X)\ar[r]&\beta_{p-1,q}^S(X)/\beta_{p-1,q}^S(X_{\texttt{q}}) \qquad\qquad\\
}
\caption{How to connect $\beta_\ast(X)$ with $\beta_\star(X_{\texttt{q}})$?}
\label{howto}
\end{figure}

Now, we state some general theorems for syzygies of inner projections, which are a generalization of main results in \cite{HK1}.

\begin{Thm}\label{coro_Np}
Let $X\subset\P^N$ be a subscheme defined by an ideal $I$, $\texttt{q}\in X$ be any smooth point and $s=s(\texttt{q},I)$ be the stabilization number of $K_i(\texttt{q},I)$'s.
\begin{itemize}
\item[(a)] Suppose that $I$ satisfies property $\textbf{N}_{d,p_0}$ as $R$-module for some $d\ge s$ and $p_0\ge 1$. Then, $\beta^S_{p,q}(\widetilde{K}_{s-1}(I))=0$ for any $0\le p<p_0 -1$ and $q>d$.
\item[(b)] Suppose that for some $d\ge s+1$ and $p_0\ge 1$, $\beta^S_{p,q}(\widetilde{K}_{s-1}(I))=0$ for any $0\le p<p_0$ and $q>d$. Then, $I$ satisfies property $\textbf{N}_{d,p_0}$ as $R$-module.
\item[(c)] $\reg_R(I)=\max\{\reg_S(\widetilde{K}_{s-1}(I)),s+1\}$.
\end{itemize}
\end{Thm}
\begin{proof}
(a) is a natural generalization of corollary 3.4 in \cite{HK1} and (b) can be also obtained by similar arguments (see \cite{Han} for details). For (c), let $\textsf{M}$ be the $\max\{\reg_S(\widetilde{K}_{s-1}(I)),s+1\}$. First, we see that $s\le \reg_R(I)-1$ by Proposition \ref{prop_stab} (b). Since $I$ satisfies $\textbf{N}_{\reg_R(I),\infty}$, we also see that $\reg_S(\widetilde{K}_{s-1}(I))\le\reg_R(I)$ by (a). Thus, $\textsf{M}\le\reg_R(I)$. Conversely, the fact that $\beta^S_{p,q}(\widetilde{K}_{s-1}(I))=0$ for any $p\ge0$ and $q>\textsf{M}$ implies that $I$ satisfies $\textbf{N}_{\textsf{M},\infty}$ by (b) so that $\textsf{M}\ge\reg_R(I)$.
\end{proof}

\begin{Thm}[Depth of inner projection]\label{depth_inn}
Let $X\subset\P^N$ be a nondegenerate subscheme defined by the saturated ideal $I_X$ and $\texttt{q}$ be a closed point of $X$. Suppose that the stabilization number $s(\texttt{q},I_X)=1$. Then,
\begin{equation}
\depth_R (X)=\depth_S (X_q).
\end{equation}
\end{Thm}
\begin{proof}
Almost same as the proof of theorem 4.1 in \cite{HK1}.
\end{proof}

\begin{Remk}[A condition for $s(\texttt{q})=1$]\label{suff_s_1} As we have seen in Theorem \ref{coro_Np} and \ref{depth_inn}, one of the most favorable cases is $s(\texttt{q})=1$ (in this case, $\widetilde{K}_{s-1}(I)$ coincides with a defining ideal of the image scheme $X_{\texttt{q}}$). First of all, $s(\texttt{q})=1$ if $I$ is \textit{quadratic} by Proposition \ref{prop_stab}.

Let us consider a little more refined condition. Then, we know in general
\begin{equation}\label{t<=e}
t:=\dim_k [K_1(I)_1] \le \codim_{\texttt{q}}(X,\P^N)=:e
\end{equation}
, where $X=V(I)$ is a subscheme of $\P^N$ and $\codim_{\texttt{q}}(X,\P^N)$ denotes the codimension of the component containing $\texttt{q}$ in $\P^N$. We can also give an interpretation on $t$ as
\begin{equation}
N-\dim_k T_{\texttt{q}}(\widetilde{X})\quad\{\textrm{$\widetilde{X}\supset X$ is the scheme defined by the quadrics of $I$}\}.
\end{equation}
If we assume the case of taking $\texttt{q}$ as a \textit{general} (so, smooth) point of $X$, then
\begin{center} 
\textit{the \textit{equality} condition $t=e$ of (\ref{t<=e}) is equivalent to $s(\texttt{q})=1$~,}
\end{center}
because $K_1(I)=I_{T_\texttt{q} X}=I_{TC_\texttt{q} X}$ in both assumptions, so it is by Proposition \ref{prop_stab} (b).
\end{Remk}

\begin{Remk}[Reduction via inner projections]\label{red_inner} In general, this inner projection method sometimes gives us a useful way to reduce many given problems into the situation of some small invariants (such as degree, codimension, etc) in which one might often solve them with the help of many nice properties of \textit{small world} in the same way as hyperplane section method did in classical algebraic geometry. In case of taking a hyperplane section the geometry goes into a relatively easier/well-known situation, while the complexity of defining equations/syzygies is almost the same. But, in case of taking an inner projection, the syzygies seem to go into a much simpler stage as compensating for a big payment of the complexity of the geometry. Furthermore, in contrast with outer projection, note that the reduction via inner projection also preserves \textit{$\Delta$-genus} in the sense of Fujita as same as the hyperplane reduction does. See \cite{P2} for a typical example using both hyperplane 
section and inner projection reductions in a clever way.

\begin{figure}[!hbt]
\begin{tabular}[t]{|c|c|}
\hline
Inner projection reduction & Hyperplane section reduction\\
\hline\hline
\textrm{For a general closed point $\texttt{q}\in X$,} &
\textrm{For a general hyperplane $H\subset\P^N$,}\\
\xymatrix{
X\subset\P^N\ar@{-->}[d]^{\pi_{\tt{q}}}\\
X_{\tt{q}}:=\overline{\pi_q(X\setminus \{q\})}\subset \P^{N-1}\\
} &\xymatrix{
X\subset\P^N\ar@{~>}[d]\\
X_{H}:=X\cap H\subset\P^{N-1}
}\\
\hline
$\codim(X_{\tt{q}})=\codim(X)-1$ & $\codim(X_{H})=\codim(X)$ \\
(or $\dim(X_{\tt{q}})=\dim(X)$) &(or $\dim(X_{H})=\dim(X)-1$) \\
\hline
$\deg(X_{\tt{q}})=\deg(X)-1$ &$\deg(X_{H})=\deg(X)$\\
\hline
$\Delta_{X_{\tt{q}}}=\Delta_X$ &$\Delta_{X_{H}}=\Delta_X$ \\
\hline

\end{tabular}
\caption{Two different reduction methods in projective algebraic geometry. Inner projection reduction versus Hyperplane section reduction. Here, $\Delta(X):=\deg(X)-\codim(X)-1$ for an embedded variety $X\subset\P^N$ (originally, due to Fujita for a polarized pair $(X,\mathcal{L})$).}
\label{fig_versus}
\end{figure}
\end{Remk}

\section{Proofs of main results}\label{prf_mains}

\subsection{Proof of inequality (\ref{ubd_main})}\label{proof_ineq}

In this subsection, we prove the inequality (\ref{ubd_main}) in Theorem \ref{fund_ineq}. In fact, we give a proof of the more strengthened form of the theorem as follows:

\begin{Thm}[On the first linear strand]\label{1st_str_ineq}
Let $X\subset \mathbb P^{N}$ be a nondegenerate subscheme, $I_X$ be the defining ideal of $X$, $\texttt{q}\in X$ be a closed point and $K_i(I_X)$'s be the PEIs of $I_X$ with respect to {\texttt{q}}. Set $t=\dim_k (K_1(I_X))_1$ and $e=\codim_{\texttt{q}}(X,\P^N)$ the codimension
of the component containing $\texttt{q}$ in $\P^N$. Then,
\begin{itemize}
\item[(a)] For any $p\ge 1$, the following holds
\begin{align}\label{ubd}
\beta_{p,1}(X)&\le\beta_{p,1}(X_{\texttt{q}})+\beta_{p-1,1}(X_{\texttt{q}})+{t\choose p}\le\beta_{p,1}(X_{\texttt{q}})+\beta_{p-1,1}(X_{\texttt{q}})+{e\choose p}\\
\label{lbd}
\beta_{p,1}(X)&\ge\beta_{p,1}(X_{\texttt{q}})+\beta_{p-1,1}(X_{\texttt{q}})+{t\choose p}-\beta_{p-1,2}(X_{\texttt{q}})-\beta^S_{p-1,2}(K_1(I))~.
\end{align}
When $p=1$, in particular, we have
\begin{equation}\label{eq_num_quad}
\beta_{1,1}(X)=\beta_{1,1}(X_{\texttt{q}})+{t \choose 1}\le\beta_{1,1}(X_{\texttt{q}})+{e \choose 1}~.
\end{equation}
\item[(b)] Furthermore, if $a=a(X)\ge1$, then
\begin{equation}\label{Np_eq}
\beta_{p,1}(X)~=~\beta_{p,1}(X_{\texttt{q}})+\beta_{p-1,1}(X_{\texttt{q}})+{e\choose p}\quad\textrm{for any}~p\le a
\end{equation}
holds and for the case of $p=a+1$ it holds that
\begin{equation}\label{Np_eq_2}
\beta_{a+1,1}(X)~=~\beta_{a+1,1}(X_{\texttt{q}})+\beta_{a,1}(X_{\texttt{q}})-\beta_{a,2}(X_{\texttt{q}})+{e\choose a+1}~.
\end{equation}
\end{itemize}
\end{Thm}

\begin{proof}
We prove the theorem by treating $\beta^R_{p,q}(I_X)$ instead of $\beta^R_{p,q}(R_X)$ (also for $\beta^S_{p,q}(S_{X_\texttt{q}})$). Because $\beta^R_{p,q}(R_X)=\beta^R_{p-1,q+1}(I_X)$ for all $p>0$ and any $q\in\Z$, keep in mind that from now on,
\begin{center}
 \textit{every $p$ in the proof is by one less than the $p$ in the statement.}\\
\end{center}

For simplicity, let $I$ be the defining ideal $I_X$ and $J$ be the ideal $K_0(I_X)=\widetilde{K}_0(I_X)$ defining $X_\texttt{q}$ scheme-theoretically. Being a nondegenerate subscheme, $I$ has no linear forms. Consider the commutative diagram such as:
\begin{equation}
\begin{gathered}{
\xymatrix @R=1.5pc @C=1.5pc{
0\ar[r]&J(-1)\ar[r]\ar[d]&I(-1)\ar[r]\ar[d]&I/J(-1)\ar[r]\ar[d]&0\\
0\ar[r]&\widetilde{K}_1(I)\ar[r]&I\ar[r]&I/\widetilde{K}_1(I)\ar[r]&0~,\\
}
}
\end{gathered}
\end{equation}
where the vertical maps are induced by $x_0$-multiplications.\\

Then, from the above diagram and elimination mapping cone sequence (Proposition \ref{MC_seq}), we have an induced commutative diagram as follows:

\begin{equation*}\small
\xymatrix @R=1.5pc @C=1pc{
&&&&&\stackrel{\vdots}{\textstyle\T^R_{p,2}(I)}\ar[d]&\\
&&0\ar[d]^-{}_{}&0\ar[r]&\T^S_{p-1,2} (J)\ar[r]^-{}&\T^S_{p-1,2}(I)\ar[r]^-{\upsilon}\ar[d]^{\mu}_{\times x_0}&\T^S_{p-1,2}(I/J)\ar[d]^{\phi}\\
0\ar[r]&\T^S_{p,2}(\widetilde{K}_1(I))\ar[r]&\T^S_{p,2}(I)\ar[d]\ar[r]^-{}&0\ar[r]&\T^S_{p-1,3} (\widetilde{K}_1(I))\ar[r]&\T^S_{p-1,3}(I)\ar[d]\ar[r]^-{\nu}&\T^S_{p-1,3}(I/\widetilde{K}_1(I))\\
&&\stackrel{\textstyle\T^R_{p,2}(I)}{\vdots}& &&\T^R_{p-1,3}(I)&,\\
} 
\end{equation*}
because $\T^S_{p,2}(I/\widetilde{K}_1(I))=0$ by Lemma \ref{trivial_van} and we assumed that $I$ has no linear forms.\\

By (\ref{ses_PEI_2}) and Proposition \ref{approx} we could identify $\phi$ with $\widetilde{\phi}$ in the following
$$\T^S_{p,2}\bigg(\frac{K_2(I)}{K_1(I)}(-2)\bigg)\to\T^S_{p-1,2}\bigg(\frac{\widetilde{K}_1(I)}{J}\bigg)\st{\widetilde{\phi}}\T^S_{p-1,3}\bigg(\frac{\widetilde{K}_2(I)}{\widetilde{K}_1(I)}\bigg)\to\T^S_{p-1,3}\bigg(\frac{K_2(I)}{K_1(I)}(-2)\bigg)~.$$
Since $\texttt{q}\in X$ so that $K_i(I)$ contains no units, we have $\T^S_{p,2}\big(\frac{K_2(I)}{K_1(I)}(-2)\big)=0$ so that $\widetilde{\phi}$ (therefore $\phi$ also) is a monomorphism. This implies that $\ker \phi\circ\upsilon =\ker~\upsilon$.\\

For (a), let us compare dimensions of kernels of morphisms in the commuting diagram above. For $\ker~\mu$ is a subspace of $\ker \nu\circ\mu = \ker~\phi\circ\upsilon$, we have
\begin{equation*}
\beta^R_{p,2}(I)-\beta^S_{p,2}(\widetilde{K}_1 (I))~=~\dim\ker~\mu ~\le~\dim\ker~\phi\circ\upsilon~=~\dim\ker~\upsilon~=~\beta_{p-1,2}^S(J)\nonumber
\end{equation*}
so that
\begin{align*}
\beta^R_{p,2}(I)\le&~\beta_{p-1,2}^S(J)+\beta^S_{p,2}(\widetilde{K}_1 (I))\nonumber\\
\le&~\beta_{p-1,2}^S(J)+\beta^S_{p,2}(J)+\beta^S_{p,2}(K_1(I)(-1))\nonumber~,
\end{align*}
because of a short exact sequence from (\ref{ses_PEI})
\begin{equation}\label{ses_PEI_3}
0\to J\to \widetilde{K}_1 (I)\to K_1 (I)(-1)\to 0~.
\end{equation}
Further, since the $k$-vector space $K_1(I)_1$ consists of $t$ independent linear forms, we can compute via a linear Koszul resolution
$$\beta^S_{p,2}(K_1(I)(-1))=\beta^S_{p,1}(K_1(I))={t\choose p+1}\le {e\choose p+1}$$ 
(note that always $t\le e$ ; see Remark \ref{suff_s_1}) and obtain the inequality (\ref{ubd}). 

The inequality (\ref{lbd}) comes from the following
\begin{align}\label{ker_est}
\dim\ker~\upsilon=&~\dim\ker~\nu\circ\mu ~=~\dim\ker~\mu +\dim(\im~\mu\cap\ker~\nu)\nonumber\\
\le&~\dim\ker~\mu +\dim \ker~\nu~
\end{align}
so that
\begin{equation}
\beta_{p-1,2}^S(J)\le\beta^R_{p,2}(I)-\beta^S_{p,2}(\widetilde{K}_1(I))+\beta^S_{p-1,3} (\widetilde{K}_1(I))\nonumber~
\end{equation}
or
\begin{equation}
\beta_{p-1,2}^S(J)+\beta^S_{p,2}(\widetilde{K}_1(I))-\beta^S_{p-1,3} (\widetilde{K}_1(I))\le\beta^R_{p,2}(I)\nonumber~.
\end{equation}
Once again, using the long exact sequence from (\ref{ses_PEI_3}), we also have the desired inequality
\begin{equation*}
\beta^S_{p,2}(J)+\beta^S_{p-1,2}(J)-\beta^S_{p-1,3}(J)+{t\choose p+1}-\beta^S_{p-1,2}(K_1(I))\le\beta^R_{p,2}(I)~.
\end{equation*}

When $p=0$, both inequalities (\ref{ubd}) and (\ref{lbd}) coincide and lead to the formula (\ref{eq_num_quad}).

For (b), above all, note that $a=a(X)\ge1$ means $I$ is quadratic and has property $\textbf{N}_{2,a}$. To prove the first part (\ref{Np_eq}), it is enough to show that $t=e$, $\beta_{p-1,2}^S(K_1(I))=0$, and $\beta^S_{p-1,3}(J)=0$ for any $p\le a-1$. Now that $I$ is quadratic, $s(\texttt{q})=1$. Thus, $t=e$ (see Remark \ref{suff_s_1}) so that $\beta^S_{p-1,2}(K_1(I))=0$. Moreover, by Fact \ref{inner_a(x)} we know that $X_{\texttt{q}}$ has at least property $\textbf{N}_{2,a-1}$ i.e. $\beta^S_{p-1,3}(J)=0$ for any $p\le a-1$. So, the equality (\ref{Np_eq}) is immediate from both (\ref{ubd}) and (\ref{lbd}). Furthermore, since $T^R_{a-1,3}(I)=0$ by property $\textbf{N}_{2,a}$, $\mu$ becomes surjective in case of $p=a$ and in this case the inequality of (\ref{ker_est}) becomes equal so that this gives the equality (\ref{Np_eq_2}).
\end{proof}

\begin{Remk}[Case of non-saturated ideals]\label{for_gen_ideal}
Note that Theorem \ref{1st_str_ineq} can be easily generalized for any scheme-theoretic defining ideal (not necessarily saturated) $I$ of $X$. Besides this theorem, most of results in this paper could drop the saturatedness.
\end{Remk}

As a test case, we give a following lemma using Theorem \ref{1st_str_ineq} (this was introduced as a part of so-called \textit{$K_{p,1}$-theorem} by Green for complex projective manifolds in \cite{G2} and also by \cite{NP} for a bit more general case).

\begin{Lem}[Generalized $K_{p,1}$-theorem (a)]\label{Kp,1}
Let $X^n\subset \P^{n+e}$ be any nondegenerate (possibly singular) variety of $\codim~e$. Then, we have
\begin{equation}\label{Kp,1_eq}
\beta_{p,1}(X)=0 \quad \text {for any}~ p > e~.
\end{equation}
\end{Lem}
\begin{proof}
Use induction on $e$. When $e=1$ (i.e. hypersurface), it is obvious. Suppose that (\ref{Kp,1_eq}) holds if $e\le m$ for some $m\ge1$. If $\codim(X,\P^{n+e})=m+1$, then take an inner projection of $X$ from any general point $\texttt{q}$ of $X$. By Theorem \ref{1st_str_ineq} (a), for any $p>m+1$ we have 
\[\beta_{p,1}(X)\le\beta_{p-1}(X_{\texttt{q}})+\beta_{p,1}(X_{\texttt{q}})+{m+1\choose p}=0\quad\textrm{($\because~p>\codim(X_{\texttt{q}},\P^{n+e-1})$)}\]
so that $\beta_{p,1}(X)=0$ and the proof is done. 
\end{proof}

\subsection{Proofs of Theorem \ref{sharp_ubd} and \ref{char_VMD}}\label{extr_case}

Now, we are ready to prove Theorem \ref{sharp_ubd} and other results. 

\begin{proof}[Proof of Theorem \ref{sharp_ubd}] We use induction on the homological index $p$. Set $X:=X^{(e)}$ to respect its own codimension and consider iterated inner projections from a general (so, non-singular) point and denote the Zariski closure of the image of $i$-th inner projection by $X^{(e-i)}$. Then, we have a chain of (birational) maps $\{\pi_k\}$ from $X$ to some lower codimensional variety (for example, a hypersurface $X^{(1)}$) and the associated sequence of varieties $\{X^{(e)},X^{(e-1)},\cdots,X^{(2)},X^{(1)}\}$ such as
\begin{align}\label{iter_inner}
X=X^{(e)}\stackrel{\pi}{\dasharrow} X^{(e-1)}\stackrel{\pi}{\dasharrow}\cdots \stackrel{\pi}{\dasharrow}X^{(e-i)}\stackrel{\pi}{\dasharrow}\cdots\stackrel{\pi}{\dasharrow} X^{(2)}\stackrel{\pi}{\dasharrow} X^{(1)}~.
\end{align}

$p=1$ case: Here, we reprove the classical result (known by Castelnuovo and independently by Zak) using our own reduction method via inner projections.  We start by recalling binomial identities (some variants of Vandermonde) which will be used frequently in the remaining part of our paper.
\begin{align}\label{bino_id_a}
\sum^s_{i=0} {r+i\choose i}&=\sum^s_{i=0} {r+i\choose r}={r+s+1\choose r+1}\\
\label{bino_id_c}
\sum^s_{i=0} {r+i\choose r}{s-i\choose t}&={r+s+1\choose r+t+1}\quad\textrm{for integers $s\ge t$}
\end{align} 

By the inequality (\ref{eq_num_quad}) of Theorem \ref{1st_str_ineq} (a), for any $e\ge 1$ we know 
\begin{align}\label{ind_1}\nonumber
\beta_{1,1}(X^{(e)})&\le \beta_{1,1}(X^{(e-1)})+{e\choose 1}\\ \nonumber
&\le\beta_{1,1}(X^{(e-2)})+{e-1\choose 1}+{e\choose 1}\\ 
&\qquad\vdots\\ \nonumber
&\le \beta_{1,1}(X^{(1)})+{2\choose 1}+\cdots+{e-1\choose 1}+{e\choose 1}\\ \nonumber
&\le {e+1\choose 2}\quad\textrm{by binomial identity (\ref{bino_id_a})}
\end{align}
, because $X^{(1)}$ is a hypersurface so that $\beta_{1,1}(X^{(1)})\le 1$.\\

Now, for some $m\ge 1$ suppose the induction hypothesis as follows:
\begin{equation}\label{ind_hyp}
\begin{array}{c}
\textit{``our desired upper bound (\ref{extr_bd}) holds for every nondegenerate variety}\\ 
\textit{of all $p\le m$ and of any codimension $e\ge 1$''}~.
\end{array}
\end{equation}
\bigskip

$p=m+1$ case: Using the inequality  (\ref{ubd}), we have
\begin{align}\label{ind_2}\nonumber
\beta_{m+1,1}(X^{(e)})&\le \beta_{m+1,1}(X^{(e-1)})+\beta_{m,1}(X^{(e-1)})+{e\choose m+1}\\ \nonumber
&\le \beta_{m+1,1}(X^{(e-2)})+\beta_{m,1}(X^{(e-2)})+\beta_{m,1}(X^{(e-1)})+{e-1\choose m+1}+{e\choose m+1}\\ 
&\qquad\vdots\\ \nonumber
&\le \beta_{m+1,1}(X^{(m)})+\sum_{i=m}^{e-1}\beta_{m,1}(X^{(i)})+\sum_{i=m+1}^{e}{i\choose m+1}\\ \nonumber
&\le \beta_{m+1,1}(X^{(m)})+m{e+1\choose m+2}+{e+1\choose m+2}~\textrm{by hypothesis (\ref{ind_hyp}) and (\ref{bino_id_a})}\\ \nonumber
&\le (m+1){e+1\choose m+2}
\end{align}
, because $\beta_{m+1,1}(X^{(m)})\le 0$ by Lemma \ref{Kp,1}. This completes our proof.
\end{proof}

As one of by-products of Theorem \ref{sharp_ubd}, we have the following new characterizations of varieties of minimal degree which generalize Castelnuovo's bound on quadrics to higher linear syzygy level.

\begin{Thm}[Theorem \ref{char_VMD}]
Let $X^n\subset \P^{n+e}$ be a nondegenerate variety with $e\ge 1$. Then, the following are all equivalent:
\begin{itemize}
\item[(a)] $X^n$ is a variety of minimal degree in $\P^{n+e}$.
\item[(b)] $\mathcal I_X$ is 2-regular.
\item[(c)] $a(X)\ge e$.
\item[(d)] $h^0(\P^{n+e}, \mathcal I_X(2))$=${e+1 \choose 2}$.
\item[(e)] one of $\beta_{p,1}(X)\textrm{'s}$ achieves the maximal upper bound (\ref{extr_bd}) for some $1\le p\le e$.
\item[(f)] all $\beta_{p,1}(X)\textrm{'s}$ achieve the maximal upper bound (\ref{extr_bd}).
\end{itemize}
\end{Thm}

\begin{proof}[Proof of Theorem \ref{char_VMD}] First, note that $(a)\Leftrightarrow (b)\Leftrightarrow (c)$. $(a)\Leftrightarrow (b)$ is well-known fact (e.g. see \cite{EH}) and $(b)\Leftrightarrow (c)$ also comes from so-called \textit{rigidity} of property $\textbf{N}_{2,p}$ (see \cite{EGHP1,HK1}). For the remaining part, we take an order such as $(f)\Rightarrow (e)\Rightarrow (d) \Rightarrow (b) \Rightarrow (f)$.

$(f)\Rightarrow (e)$ is trivial. To see $(e)\Rightarrow (d)$, use induction on $p$. For $p=1$, this implication is tautological. Assume that this is true for when $p\le m$ for some $m\ge 1$. If $\beta_{p,1}(X)$ meets its own maximum at $p=m+1$, then for any sequence of \textit{iterated general inner projections} $\{X=X^{(e)},X^{(e-1)},\cdots,X^{(m)}\}$ (see (\ref{iter_inner})), the computations as same as (\ref{ind_2}) force us to have $\beta_{m,1}(X^{(m)})=m$ and
\begin{align*}
\textrm{$\dim_k[K_1(I^{(i)})_1] = \codim(X^{(i)})$ for every $i\ge m+1$}
\end{align*} 
which implies that we have the stabilization $s=1$ at every reduction step from $X$ to $X^{(m)}$ (see Remark \ref{suff_s_1}). Here $I^{(e)}:=I_X$ and $I^{(i)}$ which defines $X^{(i)}$ scheme-theoretically is the elimination ideal of $I^{(i+1)}$.

Then, similarly as in (\ref{ind_1}), using the formula (\ref{eq_num_quad}) we obtain
\begin{align}\nonumber
\beta_{1,1}(X)&= \beta_{1,1}(X^{(e-1)})+{e\choose 1}\\ \nonumber
&=\beta_{1,1}(X^{(e-2)})+{e-1\choose 1}+{e\choose 1}\\ 
&\qquad\vdots\\ \nonumber
&= \beta_{1,1}(X^{(m)})+{m+1\choose 1}+\cdots+{e-1\choose 1}+{e\choose 1}={e+1\choose 2}~,\\ \nonumber
\end{align}
because $\beta_{m,1}(X^{(m)})=m$ implies $\displaystyle\beta_{1,1}(X^{(m)})={m+1\choose 2}$ by induction hypothesis.

To get $(d)\Rightarrow (b)$, take any sequence of iterated general inner projections from $X$ to a hypersurface $X^{(1)}$, $\{X=X^{(e)},X^{(e-1)},\cdots,X^{(1)}\}$. By the same argument we did for $(e)\Rightarrow (d)$, every reduction step from $X$ to a hypersurface $X^{(1)}$ has the stabilization number $s=1$ and $\beta_{1,1}(X^{(1)})=1$ which means $X^{(1)}$ is a hyperquadric (in particular $2$-regular). Now we can lift the regularity of $X^{(1)}$ up to the regularity of $X$ through Theorem \ref{coro_Np} (c). Hence, our $X$ is $2$-regular.

Finally, the part $(b)\Rightarrow (f)$ is also a fairly known fact (e.g. \cite{EH,Na}) and this completes the proof.
\end{proof}
\begin{Remk}[Geometric description of VMDs] Classically, the geometric classification of VMD has been known as \textit{del Pezzo-Bertini} classification. It says that every VMD which is not a linear space is either a hyperquadric, a rational normal scroll, or a cone over the Veronese surface in $\P^5$. For a modern treatment, see \cite{EH}.
\end{Remk}

\subsection{Remarks for the proofs}\label{direct_cal}

It seems to be worthwhile to write down the calculations in the proof of Theorem \ref{sharp_ubd} rather than to do it over through proof by induction. It makes one to see \textit{how one could obtain such an upper bound (\ref{extr_bd})} more clearly and gives some inspiration for the \textit{next-to-extremal case}.\\

Let us begin by meditating the formula (\ref{ubd}) a bit more. For any associated sequence of iterated general inner projections $\{X=X^{(e_0)},\cdots,X^{(e)},\cdots,X^{(1)}\}$, this formula (\ref{ubd}) says to us that 
\begin{equation}
\beta_{p,1}(X^{(e)}) \le \beta_{p-1,1}(X^{(e-1)})+\beta_{p,1}(X^{(e-1)})+{e\choose p}
\end{equation}
for every pair $(e,p)$. Figuratively speaking, \textit{one father} (i.e. $\beta_{p,1}(X^{(e)})$) has \textit{two sons} (i.e. $\beta_{p-1,1}(X^{(e-1)})$ and $\beta_{p,1}(X^{(e-1)})$) and leaves an \textit{inheritance} (i.e. $\displaystyle{e \choose p}$) to them. 

For instance, if we keep on doing this from the forefather $\beta_{p_0,1}(X^{(e_0)})$ (for simplicity, denote it by $\beta_{p_0,1}^{(e_0)}$) to fourth generation under, they become such a family and have the inheritance as appeared in Figure \ref{4gens} (page \pageref{4gens}). Here, the forefather's \textit{worth} (i.e. the value of Betti number) can be counted as the worth of all his descendants in \textit{last} (so, fourth) generation and \textit{all} the inheritances they left up to that time.

\begin{figure}[!htbp]
\begin{align*}
\renewcommand{\arraystretch}{2}
{\setlength\arraycolsep{-5pt}
\begin{array}{ccccccc}
&&&\beta_{p_0,1}^{(e_0)}&&&\\
&&\beta_{p_0 -1,1}^{(e_0-1)}&&\beta_{p_0,1}^{(e_0-1)}&&\\
&\beta_{p_0 -2,1}^{(e_0 -2)}&& 2\beta_{p_0 -1,1}^{(e_0-2)}&&\beta_{p_0,1}^{(e_0-2)}&\\
\beta_{p_0 -3,1}^{(e_0 -3)}&&3\beta_{p_0 -2,1}^{(e_0-3)}&&3\beta_{p_0 -1,1}^{(e_0-3)}&&\beta_{p_0,1}^{(e_0-3)}
\end{array}}&&&
\renewcommand{\arraystretch}{2}
{\setlength\arraycolsep{-5pt}
\begin{array}{ccccccc}
\\
&&&{e_0\choose p_0}&&&\\
&&{e_0 -1\choose p_0-1}&&{e_0-1\choose p_0}&&\\
&{e_0 -2\choose p_0-2}&& 2{e_0-2\choose p_0-1}&&{e_0-2\choose p_0}&
\end{array}}
\end{align*}
\caption{Four generations of Betti numbers (on the Left side) and their inheritances (on the Right side). Note that both of them form Pascal's triangle.}
\label{4gens}
\end{figure}

Since $\beta_{p,1}^{(e)}=0$ for any pair $(e,p)$ such that $p\le 0$ or $p>e$ (Lemma \ref{Kp,1} (a)), let us continue this \textit{Birth-Inheritance Game} (see Figure \ref{birth_inher_game} of page \pageref{birth_inher_game}) 
till all the $\beta_{p_0,1}^{(p_0)},\beta_{p_0-1,1}^{(p_0-1)},\cdots,\beta_{1,1}^{(1)}$ on the diagonal appear. Then, we can bound $\beta_{p_0,1}^{(e)}$ as follows:
\begin{align*}
\beta_{p_0,1}^{(e_0)}&\le \underbrace{\sum_{i=0}^{p_0-1}{e_0-p_0-1+i\choose i}\cdot\beta_{p_0-i,1}^{(p_0-i)}}_{(A)}&+ \underbrace{\sum_{i=0}^{p_0-1}\left\{\sum_{j=0}^{e_0-p_0-1} {i+j\choose i}{e_0-i-j\choose p_0-i}\right\}}_{(B)}
\end{align*}
, where $(A)$ corresponds to the sum of diagonal Betti numbers in Figure \ref{birth_inher_game} and $(B)$ corresponds to all the inheritance there (i.e. the sum of \textit{bold-faced} binomial numbers in the lower parallelogram).

Now, if we do this game once more on the Betti numbers on $(A)$, it follows that
\begin{align}\label{bd_by_AB1}
(A)&\le \sum_{i=0}^{p_0-1}{e_0-p_0+i\choose i}{p_0-i\choose p_0-i}={e_0\choose e_0-p_0+1}=:(A)^\prime
\end{align}
and
\begin{align}\label{bd_by_AB2}
\beta_{p_0,1}^{(e_0)}\le(A)+(B)&\le (A)^\prime+(B)= \sum_{i=0}^{p_0-1}\left\{\sum_{j=0}^{e_0-p_0} {i+j\choose i}{e_0-i-j\choose p_0-i}\right\}\\ \nonumber
&=\sum_{i=0}^{p_0-1}{e_0+1\choose p_0+1}=p_0{e_0+1\choose p_0+1}
\end{align}
by the binomial identities (\ref{bino_id_a}) and (\ref{bino_id_c}). Hence, we obtain the desired upper bounds, which represent the Betti numbers of VMD.\\

\definecolor{yqqqqq}{rgb}{0.5,0,0}
\definecolor{qqqqff}{rgb}{0,0,1}
\definecolor{yqyqyq}{rgb}{0.5,0.5,0.5}
\definecolor{cqcqcq}{rgb}{0.75,0.75,0.75}

\begin{figure}[!htbp]
\begin{tikzpicture}[line cap=round,line join=round,>=triangle 45,x=0.85cm,y=0.85cm]
\clip(-1.25,0.74) rectangle (15.83,13.36);
\draw [line width=0.4pt,dash pattern=on 4pt off 4pt,color=yqyqyq,domain=-1.25:15.83] plot(\x,{(7--8*\x)/10});
\draw [line width=1pt,dash pattern=on 4pt off 6pt] (0,0.74) -- (0,13.36);
\draw [line width=1pt,dash pattern=on 4pt off 6pt,domain=-1.25:15.83] plot(\x,{(--10.08-0*\x)/8.4});
\draw (12.3,13.22) node[anchor=north west] {$``e=p"~\textit{line}$};

\draw (5.23,13.1) node[anchor=north west] {\small$\beta^{(e_0)}_{p_0,1}$};
\draw (3.77,11.96) node[anchor=north west] {\small$\beta^{(e_0 -1)}_{p_0 -1,1}$};
\draw (5.23,11.96) node[anchor=north west] {\small$\beta^{(e_0 -1)}_{p_0,1}$};
\draw (5.23,10.78) node[anchor=north west] {\small$\beta^{(e_0 -2)}_{p_0,1}$};
\draw (3.61,10.76) node[anchor=north west] {\small$2\beta^{(e_0 -2)}_{p_0 -1,1}$};
\draw (2.19,10.76) node[anchor=north west] {\small$\beta^{(e_0 -2)}_{p_0 -2,1}$};
\draw (0.7,9.52) node[anchor=north west] {\small$\beta^{(e_0 -3)}_{p_0 -3,1}$};
\draw (2.11,9.52) node[anchor=north west] {\small$3\beta^{(e_0 -3)}_{p_0 -2,1}$};
\draw (3.65,9.52) node[anchor=north west] {\small$3\beta^{(e_0 -3)}_{p_0 -1,1}$};
\draw (5.23,9.54) node[anchor=north west] {\small$\beta^{(e_0 -3)}_{p_0,1}$};
\draw (5.3,7.18) node[anchor=north west] {\small$c_0\beta^{(p_0 +1)}_{p_0,1}$};
\draw (5.35,6) node[anchor=north west] {\small$c_0\beta^{(p_0)}_{p_0,1}$};
\draw (3.6,6) node[anchor=north west] {\small$c_1\beta^{(p_0)}_{p_0-1,1}$};
\draw (3.6,4.76) node[anchor=north west] {\small$c_1\beta^{(p_0-1)}_{p_0-1,1}$};
\draw (1.9,4.74) node[anchor=north west] {\small$c_2\beta^{(p_0 -1)}_{p_0-2,1}$};
\draw (2,3.52) node[anchor=north west] {\small$c_2\beta^{(p_0-2)}_{p_0-2,1}$};
\draw (0.3,3.52) node[anchor=north west] {\small$c_3\beta^{(p_0-2)}_{p_0-3,1}$};
\draw (0.3,2.36) node[anchor=north west] {\small$c_3\beta^{(p_0-3)}_{p_0-3,1}$};

\draw (14.49,11.96) node[anchor=north west] {\small$\mathbf{{e_0\choose p_0}}$};
\draw (14.33,10.78) node[anchor=north west] {\small$\mathbf{{e_0 -1\choose p_0}}$};
\draw (12.91,10.76) node[anchor=north west] {\small$\mathbf{{e_0 -1\choose p_0-1}}$};
\draw (12.73,9.64) node[anchor=north west] {\small$\mathbf{2{e_0 -2\choose p_0-1}}$};
\draw (14.29,9.6) node[anchor=north west] {\small$\mathbf{{e_0 -2\choose p_0}}$};
\draw (11.35,9.6) node[anchor=north west] {\small$\mathbf{{e_0 -2\choose p_0-2}}$};
\draw (14.37,8.38) node[anchor=north west] {\small$\mathbf{{e_0 -3\choose p_0}}$};
\draw (12.77,8.4) node[anchor=north west] {\small$\mathbf{3{e_0 -3\choose p_0-1}}$};
\draw (11.15,8.4) node[anchor=north west] {\small$\mathbf{3{e_0 -3\choose p_0-2}}$};
\draw (9.73,8.38) node[anchor=north west] {\small$\mathbf{{e_0 -3\choose p_0-3}}$};
\draw (14.35,7.18) node[anchor=north west] {\small$\mathbf{{e_0 -4\choose p_0}}$};
\draw (12.75,7.2) node[anchor=north west] {\small$\mathbf{4{e_0 -4\choose p_0-1}}$};
\draw (11.17,7.2) node[anchor=north west] {\small$\mathbf{6{e_0 -4\choose p_0-2}}$};
\draw (9.57,7.2) node[anchor=north west] {\small$\mathbf{4{e_0 -4\choose p_0-3}}$};
\draw (8.17,7.18) node[anchor=north west] {\small$\mathbf{{e_0 -4\choose p_0-4}}$};
\draw (14.1,4.48) node[anchor=north west] {\small$\mathbf{c_0{p_0+1\choose p_0}}$};
\draw (12.63,3.32) node[anchor=north west] {\small$\mathbf{c_1{p_0\choose p_0-1}}$};
\draw (11.11,2.1) node[anchor=north west] {\small$\mathbf{c_2{p_0-1\choose p_0-2}}$};

\draw [color=green!40!blue](-0.63,8.6) node[anchor=north west] {$\mathbf{\vdots}$};
\draw [color=green!40!blue](0.23,1.25) node[anchor=north west] {$\mathbf{\cdots}$};
\draw [color=green!40!blue](-1.3,13.56) node[anchor=north west] {e (codim.)};
\draw [color=green!40!blue](11.95,1.4) node[anchor=north west] {p (homol. index)};
\draw [color=green!40!blue](5.61,1.2) node[anchor=north west] {\small$p_0$};
\draw [color=green!40!blue](3.79,1.3) node[anchor=north west] {\small$p_0 -1$};
\draw [color=green!40!blue](2.37,1.3) node[anchor=north west] {\small$p_0 -2$};
\draw [color=green!40!blue](-0.91,10.6) node[anchor=north west] {\small$e_0 -2$};
\draw [color=green!40!blue](-0.91,11.78) node[anchor=north west] {\small$e_0 -1$};
\draw [color=green!40!blue](-0.67,12.82) node[anchor=north west] {\small$e_0$};
\draw [color=green!40!blue](1.07,1.3) node[anchor=north west] {\small$p_0 -3$};
\draw [color=green!40!blue](-0.89,9.42) node[anchor=north west] {\small$e_0 -3$};

\draw (5.61,8.22) node[anchor=north west] {$\vdots$};
\draw (4.13,8.22) node[anchor=north west] {$\vdots$};
\draw (2.89,8.26) node[anchor=north west] {$\vdots$};
\draw (1.69,8.28) node[anchor=north west] {$\vdots$};
\draw (1.69,7.06) node[anchor=north west] {$\vdots$};
\draw (2.93,7.08) node[anchor=north west] {$\vdots$};
\draw (4.09,7.1) node[anchor=north west] {$\vdots$};
\draw (2.87,5.88) node[anchor=north west] {$\vdots$};
\draw (1.65,5.9) node[anchor=north west] {$\vdots$};
\draw (1.61,4.64) node[anchor=north west] {$\vdots$};
\draw (14.91,5.9) node[anchor=north west] {\bf$\vdots$};
\draw (13.63,5.86) node[anchor=north west] {\bf$\vdots$};
\draw (12.07,5.86) node[anchor=north west] {\bf$\vdots$};
\draw (10.41,5.86) node[anchor=north west] {\bf$\vdots$};
\draw (7.09,4.66) node[anchor=north west] {\bf$\vdots$};
\draw (8.77,4.68) node[anchor=north west] {\bf$\vdots$};
\draw (10.43,4.7) node[anchor=north west] {\bf$\vdots$};
\draw (13.57,4.78) node[anchor=north west] {\bf$\vdots$};
\draw (10.39,3.44) node[anchor=north west] {\bf$\vdots$};
\draw (8.77,3.44) node[anchor=north west] {\bf$\vdots$};
\draw (8.77,2.26) node[anchor=north west] {\bf$\vdots$};
\draw (7.09,3.5) node[anchor=north west] {\bf$\vdots$};
\draw (7.09,2.26) node[anchor=north west] {\bf$\vdots$};
\draw (0.41,8.26) node[anchor=north west] {$\cdots$};
\draw (0.37,7.02) node[anchor=north west] {$\cdots$};
\draw (0.37,5.8) node[anchor=north west] {$\cdots$};
\draw (0.37,4.6) node[anchor=north west] {$\cdots$};
\draw (5.61,3.44) node[anchor=north west] {\bf$\vdots$};
\draw (5.61,2.28) node[anchor=north west] {\bf$\vdots$};
\draw (4.19,2.26) node[anchor=north west] {\bf$\vdots$};
\draw (8.77,5.84) node[anchor=north west] {\bf$\vdots$};
\draw (12.07,4.72) node[anchor=north west] {\bf$\vdots$};
\draw (12.07,3.54) node[anchor=north west] {\bf$\vdots$};
\draw (10.37,2.26) node[anchor=north west] {\bf$\vdots$};
\draw (8.77,1.36) node[anchor=north west] {\bf$\vdots$};
\draw (7.09,1.36) node[anchor=north west] {\bf$\vdots$};
\draw (10.37,1.36) node[anchor=north west] {\bf$\vdots$};
\draw (-0.83,7.02) node[anchor=north west] {$\cdots$};
\draw (-0.83,5.8) node[anchor=north west] {$\cdots$};
\draw (-0.83,4.6) node[anchor=north west] {$\cdots$};
\draw (-0.87,3.42) node[anchor=north west] {$\cdots$};
\draw (-0.87,2.24) node[anchor=north west] {$\cdots$};
\end{tikzpicture}
\caption{\textsf{Birth-Inheritance Game}. People (i.e. \textit{Betti number} $\beta_{p,1}^{(e)}\textrm{'s}$) are located, according to the pair $(e,p)$, in the upper triangular area and all their inheritance (i.e. \textbf{bold-faced} \textit{binomial numbers}) are stacked up in the lower triangular area. Each person gives birth to two sons and leaves the inheritance until one reaches the diagonal (i.e. \textit{$e=p$ line}). Note that all the inheritances form a parallelogram and the coefficient $c_i={e_0-p_0-1+i\choose i}$.}
\label{birth_inher_game}
\end{figure}
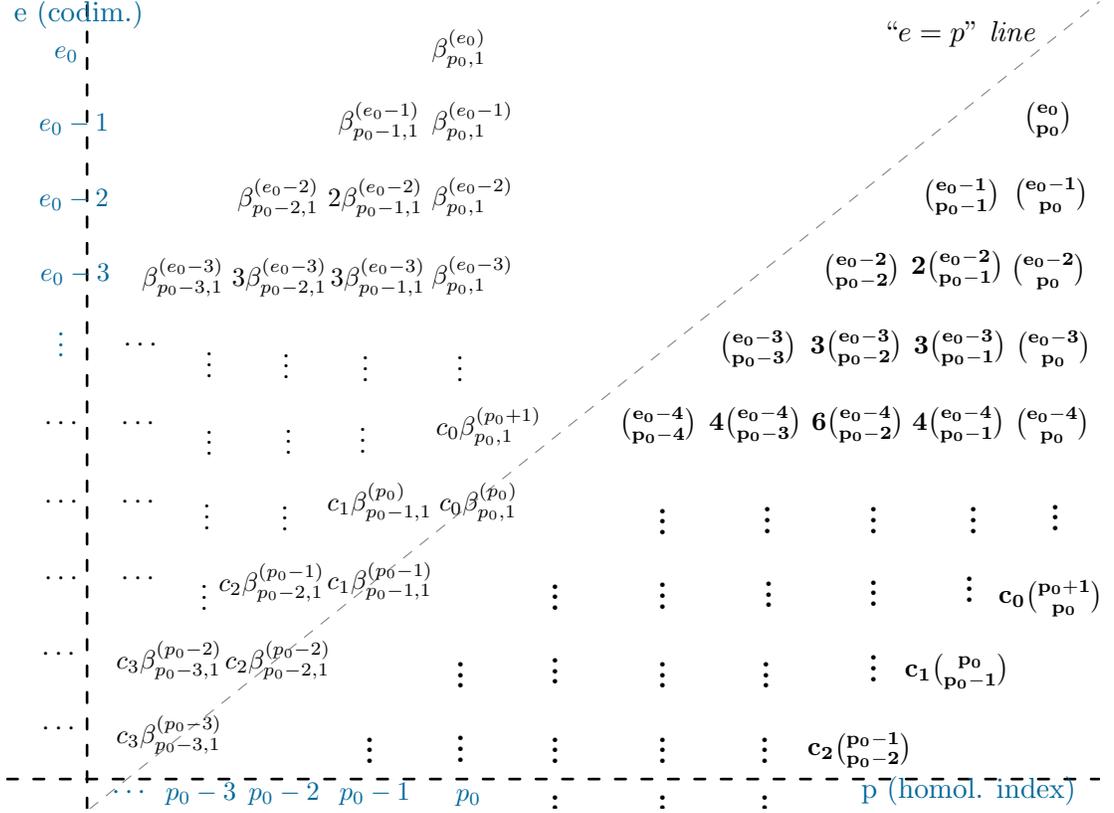

\bigskip

\section{Next-to-extremal case}\label{n_extr}

\begin{Thm}\label{n_extr_case}
Let $X^n\subset\P^{n+e}$ be any nondegenerate variety of codim $e$. Unless $X$ is a variety of minimal degree, then we have
\begin{equation}\label{n_extr_bd}
\beta_{p,1}(X)\leq \left\{
\begin{array}{cl} \displaystyle p{e+1 \choose p+1}-{e\choose p-1} &(1\le p\le e)\\
&\\
\displaystyle 0 & (p>e)\quad.
\end{array}\right.
\end{equation}
\end{Thm}

Note that $\displaystyle p{e+1 \choose p+1}-{e\choose p-1}$ is also the $p$-th Betti number of del Pezzo varieties of codimension $e$ (e.g. \cite{Na}). Before proving Theorem \ref{n_extr_case}, we introduce another relevant lemma as a direct consequence of theorem 3.5 in \cite{NP}.

\begin{Lem}[Generalized $K_{p,1}$-theorem (b)]\label{Kp,1_b}
Let $X^n\subset\P^{n+e}$ be any nondegenerate variety of codim $e$. Unless $X$ is a variety of minimal degree, then we have
\[\beta_{e,1}(X)=0~.\]
\end{Lem}

Now, let's prove next-to-extremal upper bounds on $\beta_{p,1}$'s.

\begin{proof}[Proof of Theorem \ref{n_extr_case}]
First, we note that a general inner projection $X_{\texttt{q}}$ is not of minimal degree, unless $X$ is of minimal degree (due to so-called \textit{Trisecant lemma}). Similarly as in the proof of extremal bounds, take a sequence of iterated general inner projections $\{X=X^{(e)},X^{(e-1)},\cdots,X^{(1)}\}$. As discussed in subsection \ref{direct_cal}, we could bound
\begin{align*}
\beta_{p,1}(X)&\le \underbrace{\sum_{i=0}^{p-1}{e-p-1+i\choose i}\cdot\beta_{p-i,1}(X^{(p-i)})}_{(A)}+ \underbrace{\sum_{i=0}^{p-1}\left\{\sum_{j=0}^{e-p-1} {i+j\choose i}{e-i-j\choose p-i}\right\}}_{(B)}\\
&=(B)=p{e+1\choose p}-(A)^\prime=p{e+1\choose p} -{e-1\choose p}\quad\textrm{(see (\ref{bd_by_AB1}) and (\ref{bd_by_AB2})).}
\end{align*}
, because $\beta_{p-i,1}(X^{(p-i)})=0$ for every $0\le i\le p-1$ by Lemma \ref{Kp,1_b}.
\end{proof}

As an application, we can also add new characterizations of del Pezzo varieties which generalize Fano's classical bound on quadrics to higher linear syzygy level.

\begin{Thm}\label{char_del Pezzo}
Let $X^n\subset \P^{n+e}$ be a nondegenerate variety with $e\ge2$. Then, the following are all equivalent:
\begin{itemize}
\item[(a)] $X$ is a del Pezzo variety.
\item[(b)] $a(X)=e-1$.
\item[(c)] $h^0(\P^{n+e}, \mathcal I_X(2))$=${e+1 \choose 2}-1$.
\item[(d)] one of $\beta_{p,1}(X)\textrm{'s}$ achieves the upper bound (\ref{n_extr_bd}) for some $1\le p\le e-1$.
\item[(e)] all $\beta_{p,1}(X)\textrm{'s}$ achieve the upper bound (\ref{n_extr_bd}).
\end{itemize}
\end{Thm}
\begin{proof}
$(a)\Leftrightarrow(b)$ is known by theorem 4.3 (b) in \cite{HK1} and we prove by taking an order such as $(b)\Rightarrow(e)\Rightarrow(d)\Rightarrow(c)\Rightarrow(a)$.

$(b)\Rightarrow(e)$ comes from the known Betti numbers of del Pezzo varieties (e.g. \cite{Na}) and $(e)\Rightarrow(d)$ is trivial. Now let us see $(d)\Rightarrow(c)$. As seen in the proof of Theorem \ref{n_extr_case}, the equality of next-to-extremal bound on some $\beta_{p,1}(X)$ means that every reduction step from $X=X^{(e)}$ to $X^{(1)}$ for any sequence of iterated general inner projections $\{X=X^{(e)},X^{(e-1)},\cdots,X^{(1)}\}$ should have the stabilization $s=1$ and $\beta_{1,1}(X^{(1)})=0$. Thus, using the formula (\ref{eq_num_quad}) repeatedly, we obtain $\beta_{1,1}(X)={e+1\choose 2}-1$. Finally, to show $(c)\Rightarrow(a)$ note that the delta genus is preserved (see Figure \ref{fig_versus}) under each reduction (i.e. $\Delta(X^{(i+1)})=\Delta(X^{(i)})$ for every $i\ge1$) and that $X^{(2)}$ is a complete intersection of two quadrics. Since $X^{(2)}$ is a variety of next-to-minimal degree (i.e. $\Delta=1$) and ACM, we conclude that our original $X$ is also of next-to-minimal degree and ACM (\textit{depth} can be lifted by 
Theorem \ref{depth_inn} whenever $s=1$), in other words a del Pezzo variety.
\end{proof}

\begin{Remk}[Geometric characterization of del Pezzo]
Some works on the geometric characterization/classification of del Pezzo varieties have been done by Fujita for mainly normal singularities and recently by Brodmann and Park for non-normal cases (see Remark 4.4 (b) in \cite{HK1} for references). 
\end{Remk}

\section{Examples and questions}\label{ex_qu}

\paragraph{\bf More general categories} As we explored through Theorem \ref{sharp_ubd} and \ref{char_VMD}, in the category of $k$-varieties $\mathsf{Var}(k)$ all the notions \textit{minimal degree, 2-regularity,} and \textit{maximal Betti numbers} are same. How about more general categories?

In \cite{EGHP2} they appointed `2-regularity' as a generalization of the notion of `minimal degree', clarified its geometric meaning (so-called \textit{smallness}), and classify them completely in the category of algebraic sets $\mathsf{AlgSet}(k)$. We could also attempt to extend the notion of `maximal Betti numbers' and to generalize similar characterizations on them into more general categories (even though not into the whole $\mathsf{AlgSet}(k)$).

For instance, let us consider the following category. One says that any algebraic set $X=\cup X_i$ is \textit{connected in codimension 1} if $X$ is equidimensional and all the irreducible component $X_i$'s can be ordered in such a way that every $X_i\cap X_{i+1}$ is of codimension 1 in $X$. Denote the category of connected in codimension 1 algebraic sets by $\mathsf{CC}_1(k)$.

Note that a key ingredient for proofs of most of results in this paper is the reduction method via inner projections where the notion of \textit{codimension} has an important role. $\mathsf{CC}_1(k)$ is the very case in which codimension is well-defined (and the \textit{degree} is always at least codimension+1) and reduction process are well-behaved as following steps (see also Figure \ref{red_comp}): 
\begin{itemize}
\item[i)] choose one component and take iterated general inner projections within the component until the component disappear (into the intersection with other components)
\item[ii)] do these reductions component by component. 
\end{itemize}
{\setlength\abovecaptionskip{.1ex plus .125ex minus .125ex}
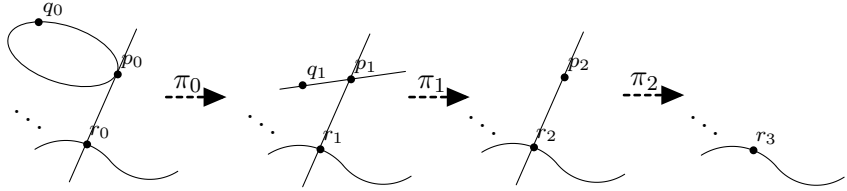
\begin{figure}[!htbp]
\begin{tikzpicture}[line cap=round,line join=round,>=triangle 45,x=0.6571741511500546cm,y=0.6557377049180326cm]
\clip(-4.3,1.42) rectangle (13.96,6.3);
\draw [rotate around={-18.07:(-2.63,4.87)}] (-2.63,4.87) ellipse (0.76cm and 0.38cm);
\draw (-1.14,5.38)-- (-2.5,2.28);
\draw [shift={(-2.57,1.94)}] plot[domain=0.72:2.13,variable=\t]({1*1.18*cos(\t r)+0*1.18*sin(\t r)},{0*1.18*cos(\t r)+1*1.18*sin(\t r)});
\draw [shift={(-0.87,3.38)}] plot[domain=3.83:5.28,variable=\t]({1*1.05*cos(\t r)+0*1.05*sin(\t r)},{0*1.05*cos(\t r)+1*1.05*sin(\t r)});
\draw (-3.88,4.32) node[anchor=north west] {$\mathbf{\ddots}$};
\draw [->,line width=1pt,dash pattern=on 2pt off 2pt] (-0.54,4) -- (0.68,4);
\draw (-0.6,4.64) node[anchor=north west] {$\pi_0$};
\draw (3.58,5.28)-- (2.22,2.18);
\draw [shift={(2.15,1.84)}] plot[domain=0.72:2.13,variable=\t]({1*1.18*cos(\t r)+0*1.18*sin(\t r)},{0*1.18*cos(\t r)+1*1.18*sin(\t r)});
\draw [shift={(3.85,3.28)}] plot[domain=3.83:5.28,variable=\t]({1*1.05*cos(\t r)+0*1.05*sin(\t r)},{0*1.05*cos(\t r)+1*1.05*sin(\t r)});
\draw (0.84,4.22) node[anchor=north west] {$\mathbf{\ddots}$};
\draw [->,line width=1pt,dash pattern=on 2pt off 2pt] (4.38,4) -- (5.6,4);
\draw (4.34,4.64) node[anchor=north west] {$\pi_1$};
\draw (1.76,4.16)-- (4.3,4.52);
\draw (7.9,5.32)-- (6.54,2.22);
\draw [shift={(6.47,1.88)}] plot[domain=0.72:2.13,variable=\t]({1*1.18*cos(\t r)+0*1.18*sin(\t r)},{0*1.18*cos(\t r)+1*1.18*sin(\t r)});
\draw [shift={(8.17,3.32)}] plot[domain=3.83:5.28,variable=\t]({1*1.05*cos(\t r)+0*1.05*sin(\t r)},{0*1.05*cos(\t r)+1*1.05*sin(\t r)});
\draw (5.31,4.2) node[anchor=north west] {$\mathbf{\ddots}$};
\draw [->,line width=1pt,dash pattern=on 2pt off 2pt] (8.72,4.04) -- (9.94,4.04);
\draw (8.64,4.7) node[anchor=north west] {$\pi_2$};
\draw [shift={(10.91,1.82)}] plot[domain=0.72:2.13,variable=\t]({1*1.18*cos(\t r)+0*1.18*sin(\t r)},{0*1.18*cos(\t r)+1*1.18*sin(\t r)});
\draw [shift={(12.61,3.26)}] plot[domain=3.83:5.28,variable=\t]({1*1.05*cos(\t r)+0*1.05*sin(\t r)},{0*1.05*cos(\t r)+1*1.05*sin(\t r)});
\draw (9.8,4.2) node[anchor=north west] {$\mathbf{\ddots}$};
\begin{scriptsize}
\fill [color=black] (-2.14,3.04) circle (1.5pt);
\draw[color=black] (-1.88,3.3) node {$r_0$};
\fill [color=black] (-1.52,4.46) circle (1.5pt);
\draw[color=black] (-1.22,4.72) node {$p_0$};
\fill [color=black] (-3.12,5.52) circle (1.5pt);
\draw[color=black] (-2.82,5.78) node {$q_0$};
\fill [color=black] (2.58,2.94) circle (1.5pt);
\draw[color=black] (2.84,3.2) node {$r_1$};
\fill [color=black] (3.2,4.36) circle (1.5pt);
\draw[color=black] (3.5,4.62) node {$p_1$};
\fill [color=black] (2.22,4.24) circle (1.5pt);
\draw[color=black] (2.52,4.5) node {$q_1$};
\fill [color=black] (6.9,2.98) circle (1.5pt);
\draw[color=black] (7.16,3.24) node {$r_2$};
\fill [color=black] (7.52,4.4) circle (1.5pt);
\draw[color=black] (7.82,4.66) node {$p_2$};
\fill [color=black] (11.34,2.92) circle (1.5pt);
\draw[color=black] (11.6,3.18) node {$r_3$};
\end{scriptsize}
\end{tikzpicture}
 \caption{How to reduce components following i) and ii) in $\mathsf{CC}_1(k)$. The dashed arrows represent inner projections $\pi_i$'s from $\textsf{q}_0$, $\textsf{q}_1$, and $\textsf{p}_2$ respectively. Note that every reduction step diminishes codimension exactly by one.}
\label{red_comp}
\end{figure}}

Therefore, our extremal bounds and characterizations for the maximal Betti numbers in $\mathsf{Var}(k)$ can be naturally generalized to this category $\mathsf{CC}_1(k)$. 
\bigskip
\begin{Thm}\label{cc1_sharp_ubd}
Let $X^n\subset\P^{n+e}$ be any nondegenerate algebraic set of codim $e\ge 1$ in $\mathsf{CC}_1(k)$. Then,
\begin{equation}\label{cc1_extr_bd}
\beta_{p,1}(X)\leq p{e+1 \choose p+1}\quad \textrm{for}~all~p\ge 0~.
\end{equation}
 Further, the following are all equivalent:
\begin{itemize}
\item[(a)] $X$ is of minimal degree in $\P^{n+e}$.
\item[(b)] $\mathcal I_X$ is 2-regular.
\item[(c)] $a(X)\ge e$.
\item[(d)] $h^0(\P^{n+e}, \mathcal I_X(2))$=${e+1 \choose 2}$.
\item[(e)] one of $\beta_{p,1}(X)\textrm{'s}$ achieves the maximum for some $1\le p\le e$.
\item[(f)] all $\beta_{p,1}(X)$ achieve the maxima.
\end{itemize}
\end{Thm}

\begin{Remk}
In $\mathsf{CC}_1(k)$, we can see easily who the maximal Betti numbers are geometrically. First, we recall that a sequence $\{X_1,X_2,\cdots,X_n\}$ of the components of an algebraic set $X=\cup X_i$ is \textit{linearly joined} if we have 
\[(X_1\cup\cdots\cup X_i)\cap X_{i+1}=\langle X_1\cup\cdots\cup X_i\rangle\cap \langle X_{i+1}\rangle\] 
for every $i=1,2,\cdots,n-1$, where $\langle X_i\rangle$ means its span (so this may depend on the ordering). Being of minimal degree, they are just all the linearly joined union of VMDs. This also coincides with the classification of \cite{EGHP2}, because of 2-regularity. 
\end{Remk}

Now, we look some interesting examples up. Since the theory is closed related to the geometry of codimension, the examples have chosen among the curve cases.

\begin{Ex}
Let $X_1\subset\P^4$ be a union of a line $\ell$ and a twisted cubic $C$ such that $\ell\cap C=\langle\ell\rangle\cap\langle C\rangle=one~point$. Let $X_2$ be a union of two plane conics $Q_1, Q_2$ meeting $one~point$ (their spans also) in $\P^4$. Both $X_1$ and $X_2$ in $\mathsf{CC}_1(k)$ are linearly joined unions of VMDs and $\codim~e=3$. Using \texttt{Macaulay 2} (see \cite{M2}), we can verify that they give the same Betti table having maximal Betti numbers as expected in Theorem \ref{cc1_sharp_ubd}.
\[
\begin{array}{c|c|c|c|c|}
   &0 &1 &2 &3 \\ \hline
0 &\tt{1} & - &- &- \\ \hline
1 &- &\tt{6} &\tt{8} &\tt{3} \\ \hline
\end{array}\\
\]
\end{Ex}

But, we can not drop the condition `connected in codimension 1' in Theorem \ref{cc1_sharp_ubd} as the following example says.

\begin{Ex}[Skew lines in $\P^3$]\label{skew} Let $X=\ell_1\cup \ell_2$ be skew lines in $\P^3$ ($e=2$) and set $I_X=(x_0 x_2,x_0 x_3,x_1 x_2,x_1 x_3)$. $X$ is nondegenerate, linearly joined, but not connected in codimension 1 (by convention, consider $\dim\emptyset=-1$). By \texttt{Macaulay 2}, we can compute the Betti table of $X$ as below. All $\beta_{p,1}$'s \textit{exceed} the extremal bounds in (\ref{cc1_extr_bd}) of codimension 2.
\[
\begin{array}{c|c|c|c|c|}
   &0 &1 &2 &3 \\ \hline
0 &\tt{1} & - &- &- \\ \hline
1 &- &\tt{4} &\tt{4} &\tt{1} \\ \hline
\end{array}\]
\end{Ex}

We also have examples which show that the bounds (\ref{n_extr_bd}) may not serve as next-to-extremal bounds in $\mathsf{CC}_1(k)$. In other words, from the consideration of next-to-extremal case it might be possible to occur many interesting Betti tables according to the configurations of unions of small degree varieties even though in the category $\mathsf{CC}_1(k)$.

\begin{Ex}[On next-to-extremal bound]\label{prb_n_ext} Let $X_1\subset\P^4$ be a union of a plane conic $Q$ and a twisted cubic $C$ meeting $one~double~point$ with $\langle Q\rangle\cap\langle C\rangle=\P^1$ ($e=3$). This is nondegenerate, connected in codimension 1, but not linearly joined. $X_1$ is also of next-to-minimal degree and has the same Betti table as a del Pezzo variety does in $\mathsf{Var}_1(k)$ (see Figure \ref{2tables}). On the other hand, if $X_2$ is a nondegenerate union of plane cubic $C$ (i.e. a singular projection of twisted cubic) and a plane conic $Q$ in $\P^4$ ($e=3$), then $X_2$ has a different Betti table with the one of $X_1$, although $X_2$ is of next-to-minimal degree, connected in codimension 1, and even \textit{linearly joined} (see also Figure \ref{2tables}). We see that $\beta_{2,1}$ and $\beta_{3,1}$ exceed next-to-extremal bounds (\ref{n_extr_bd}) though $\beta_{1,1}$ \textit{achieves} the maximum of (\ref{n_extr_bd}). Note that two Betti tables get the same after taking a \textit{diagonal cancellation}.
\end{Ex}

\begin{figure}[!hbt]
\[\begin{array}{cccc}
\mathbb{B}(X_1)\qquad\begin{array}{c|c|c|c|c|}
   &0 &1 &2 &3 \\ \hline
0 &\tt{1} & - &- &- \\ \hline
1 &- &\tt{5} &\tt{5} &- \\ \hline
2 &- &- &- & \tt{1} \\ \hline
\end{array}
&&&
\mathbb{B}(X_2)\qquad\begin{array}{c|c|c|c|c|}
   &0 &1 &2 &3 \\ \hline
0 &\tt{1} & - &- &- \\ \hline
1 &- &\tt{5} &\tt{6} &\tt{2} \\ \hline
2 &- &\tt{1} &\tt{2} & \tt{1} \\ \hline
\end{array}
\end{array}\]
\caption{Two Betti tables of $X_1$ and $X_2$, algebraic sets of next-to-minimal degree in $\mathsf{CC}_1(k)$ by \texttt{Macaulay 2}. Note that two tables be the same after a \textit{diagonal cancellation}.}
\label{2tables}
\end{figure}

\begin{Qu} Here are our questions.
\begin{itemize}
 \item[(a)] Is it possible to generalize upper bounds (\ref{cc1_extr_bd}) and (\ref{n_extr_bd}) into more general categories such as $\mathsf{AlgSet}(k)$ (possibly in terms of codimensions of components and other invariants, if needed)?
 \item[(b)] Can we explain the reason of the difference of two Betti tables in Figure \ref{2tables} \textit{geometrically}? Is it possible to heal the next-to-extremal case in $\mathsf{CC}_1(k)$ (see Example \ref{prb_n_ext}) by figuring out this \textit{phenomena} of diagonal cancellation?  
 \item[(c)] Classify or characterize those who have \textit{next-to-simple} Betti tables (the \textit{simplest} are the tables of 2-regular schemes) geometrically in $\mathsf{CC}_1(k)$ or more general categories (see also question 5.6 in \cite{HK1}).
\end{itemize}
\end{Qu}

\paragraph{\bf More improved bounds}

Concerning on linear syzygies of $X$ at least, one could say in general 
\begin{center}
\textit{More quadrics $X$ has, Nicer syzygies $X$ has.}\\
\end{center}
Here, \textit{what `niceness' does mean} could be spoken in many different ways, but in view of Theorem \ref{char_VMD} and \ref{char_del Pezzo} we can say it means getting closer to maximal Betti numbers in the linear strand and higher $a(X)$ (or $b(X)$) our $X$ has.

On this point there is an interesting fact such as (coming from Corollary 3.8 in \cite{HK1}): 
\begin{Fact}
Let $X^n\subset \P^{n+e}$ be a nondegenerate subscheme in $\mathsf{Var}(k)$ (or $\mathsf{CC}_1(k)$) of $\codim~e$. Then, we have
\begin{equation}
{e+1\choose 2}-{e+1-a(X)\choose 2}\le\beta_{1,1}(X)~.
\end{equation}
\end{Fact}
In other words, it means that $a(X)$ has some necessary conditions on $\beta_{1,1}$.  Therefore, we suspect that the following question might be true:
\begin{equation}\label{qu_b(X)}
\textrm{Is it possible to give an upper bound on $\beta_{1,1}(X)$ in terms of $b(X)$?}
\end{equation} 
, that is the question about whether $\beta_{1,1}$ does impose some sufficient condition for $b(X)$ or not. For a large $b(X)$ (to be precise, for $b(X)\ge e$ in $\mathsf{Var}(k)$), (\ref{qu_b(X)}) is true. It is also considered as a kind of \textit{converse} of the idea, say
\begin{center}
\textit{High $b(X)$ guarantees many quadrics on $X$ so that $X$ can inherits interesting geometric structures from the embedding quadrics}
\end{center}
, on which many problems (e.g. Green's conjectures on algebraic curves in \cite{G2}) are mainly based. As one of the ways to answer the question (\ref{qu_b(X)}), we raise the following question:
\begin{Qu}\label{qu_inner_b(X)}
Let $X^n\subset \P^{n+e}$ be a nondegenerate reduced subscheme of $\codim~e$ and $X_{\texttt{q}}$ be its inner projected image. 
\begin{center}
Does it hold that $b(X_{\texttt{q}})\le b(X)-1$ for a general point $\texttt{q}\in X$?
\end{center}
\end{Qu}

\begin{Remk}\label{remk_qu_b} We complete by making some relevant remarks.
\begin{itemize}
\item[(a)] For $a(X)$, we have an interesting result from corollary 3.4 in \cite{HK1}:
\begin{Fact}\label{inner_a(x)}
Let $X^n\subset \P^{n+e}$ be a nondegenerate reduced subscheme of $\codim~e$ and $X_{\texttt{q}}$ be its inner projected image. Then, we have
$$a(X_q)\ge a(X)-1~\textrm{for a general (in fact, any smooth) point $\texttt{q}\in X$.}$$
\end{Fact}
\item[(b)] We know that $b(X_{\texttt{q}})\le b(X)$ for a general $\texttt{q}\in X$ always holds. To the best of author's knowledge, there hasn't been a counterexample for Question \ref{qu_inner_b(X)} except the case of $\texttt{q}$ being singular. If Question \ref{qu_inner_b(X)} is true, then through similar arguments in subsection \ref{direct_cal},  we can answer the question (\ref{qu_b(X)}) as follows:
\begin{align*}
\beta_{p,1}(X)&\le p{e+1\choose p+1}+\left\{{e+1\choose p+1}-{b\choose p+1}\right\}-(e-b
+1){e+1\choose p}
\end{align*}
, which are more improved upper bounds in terms of $e,p,b:=b(X)$ generalizing the bounds (\ref{extr_bd}) and (\ref{n_extr_bd}).
\end{itemize}
\end{Remk}



\end{document}